\documentclass[abstracton]{scrartcl} 

\usepackage[T1]{fontenc}

\usepackage{amsfonts}
\usepackage{amsmath}
\usepackage{amssymb}
\usepackage{xcolor}
\usepackage{dsfont}

\usepackage{graphics}
\usepackage{graphicx}
\usepackage{epstopdf}
\usepackage{epsf}

\usepackage{subcaption}
\usepackage{float}

\usepackage{latexsym}
\usepackage{oldgerm}
\usepackage{dsfont}
\usepackage{amscd, amsthm,enumerate,verbatim,calc}
\usepackage{authblk}
\usepackage{enumitem}
\usepackage{thm-restate}

\newtheorem{theorem}{Theorem}
\newtheorem{Lemma}[theorem]{Lemma}
\newtheorem{proposition}[theorem]{Proposition}
\newtheorem{corollary}[theorem]{Corollary}

\newtheorem{fact}[theorem]{Fact}
\renewenvironment{proof}{{\noindent \emph{Proof.} }}{\hfill $\Box$ \\} 

\renewcommand{\qed}{\hfill $\Box$}

\newcommand{\w}{\color{black}}

\usepackage[colorlinks,breaklinks,backref]{hyperref}
\usepackage{hyperref}
\usepackage{backref}
\usepackage{todonotes}

\usepackage{thmtools}
\usepackage{thm-restate}
\usepackage{cleveref}

\begin{document}

\title {
On $k$-colorability of $(bull,H)$-free graphs}

\author[1]{Nadzieja Hodur}
\author[1]{Monika Pil\'sniak}
\author[1]{Magdalena Prorok}
\author[1, 2]{Ingo Schiermeyer}
\affil[1]{\normalsize AGH University of Krakow, al. Mickiewicza 30, 30-059 Krak\'ow, Poland}
\affil[2]{\normalsize TU Bergakademie Freiberg, 09596 Freiberg, Germany}

\date{\today}

\maketitle
\begin{abstract}
    The $3$-colorability problem is a well-known NP-complete problem and it remains NP-complete for $bull$-free graphs, where a $bull$ is the graph consisting of a $K_3$ with two pendant edges attached to two of its vertices. In this paper, for $k\geq3$, we characterize all $k$-colorable $(bull,claw)$-free graphs containing an induced cycle of length at least $6$. Moreover, we present the full characterization of all non $4$-colorable connected  $(bull,claw)$-free graphs and $(bull,chair, C_5)$-free graphs, and all non $5$-colorable connected $(bull, claw, C_5)$-free graphs.
\end{abstract}

\noindent
{\w Keywords: $4$-colorable graphs, forbidden induced subgraphs, perfect graphs, complexity\\               
Math. Subj. Class.: { 05C15,  05C17, 68Q25, 68W40. }                      

\section{Introduction}\label{sec:intro}

We consider finite, simple, and  undirected graphs.
For terminology and notations not defined here, we refer to~\cite{BM08}.

An {\it induced subgraph} of a graph $G$
is a graph on a vertex set $S \subseteq V(G)$
for which two vertices are adjacent
if and only if they are adjacent in $G$.
In particular, we say that the subgraph is \emph{induced by $S$}.
We also say that a graph $H$ is an \emph{induced subgraph} of $G$ if $H$~is
isomorphic to an induced subgraph of $G$.

Given a family $\cal{H}$ of graphs and a graph $G$, we say that $G$ is \emph{$\cal{H}$-free}
if $G$ contains no graph from $\cal{H}$ as an induced subgraph.
In this context, the graphs of $\cal{H}$ are referred to as
\emph{forbidden induced subgraphs}.

A graph is {\it $k$-colorable} if each of its vertices can be colored with one of $k$ colors
so that adjacent vertices obtain distinct colors.
The smallest integer $k$ such that a given graph $G$ is $k$-colorable
is called the {\it chromatic number} of $G$, denoted by $\chi(G).$
Clearly, $\chi(G) \geq \omega(G)$ for every graph $G$,
where $\omega(G)$ denotes the \emph{clique number} of $G$,
that is, the order of a maximum complete subgraph of $G$.
%
Furthermore, a graph $G$ is {\it perfect} if $\chi(G')=\omega(G')$
for every induced subgraph $G'$ of $G.$
For a subgraph $H$ and a vertex $v$, let $d_H(v) = |N(v) \cap V(H)|$.

For an induced cycle $C_p$ with $p \geq 3$ let $C_p[k_{1}, k_{2}, \ldots, k_{p}]$ denote the \emph{clique expansion} of an induced cycle $C_p$, where its vertices $v_1, v_2, \ldots, v_p$ are replaced by complete graphs $K_{k_i}$ for $1 \leq i \leq p$ and additional edges between all pairs of vertices from consecutive cliques. (see Fig. \ref{spindler}). By $G \oplus H$ we denote a graph with set of vertices $V(G) \cup V(H)$ and set of edges $E(G) \cup E(H) \cup \{vw: v \in V(G), w \in V(H)\}$.

Let $G$ be a clique expansion $C[k_1, \ldots, k_n]$ of a cycle $C_n$ and let $k \in \mathbb{N}$. Let us label the vertices of the first clique with numbers $1, \ldots, k_1$, vertices of the second clique with numbers $k_1+1, \ldots, k_1+k_2$ and so on. Then the \emph{circular $k$-coloring algorithm}, called also $k$-CC algorithm, is an algorithm assigning to $m$-th vertex of $G$ the color $((m-1) \textnormal{ mod } k)+1$ for $m=1, \ldots, k_1+\ldots +k_n$.

The graph on five vertices $v_1$, $v_2$, $v_3$, $v_4$, $v_5$ and with the edges $v_1v_2$, $v_2v_3$, $v_3v_4$, $v_4v_5$, $v_2v_4$ is called a $bull$. 
Let $S_{i,j,k}$ be a $3$-star with edges subdivided respectively $i-1$, $j-1$ and $k-1$ times.
The graph $S_{1,1,1}$ is called a \emph{claw} and $S_{1,1,2}$ is called a $chair$. 

The independence number $\alpha(G)$ of the graph $G$ is the largest $k \in \mathbb{N}$ such that there exists $S \subset V(G)$, satisfying $|S| = k$ and $S$ is a set of independent vertices.

The $3$-colorability problem is a well-known NP-complete problem and it remains NP-complete for $claw$-free graphs and $K_3$-free graphs. In the last two decades, a large number of results of colorings of graphs with forbidden subgraphs have been shown (cf.~\cite{BDHKRSV22}, \cite{alpha3}, \cite{clawdiamondnet}, \cite{Ran}, \cite{RS04-1}, \cite{RST02}, \cite{Sum} and cf.~\cite{GJPS17}, \cite{RS04}, \cite{RS19} for three surveys).

Our research has been motivated by~\cite{clawdiamondnet} and we use some definitions and notations from it. A graph $G$ of order $3p+1$, $p\geq1$ is called a \emph{spindle graph} $M_{3p+1}$ if it contains a~cycle $C$: $u_0 u_1 \ldots  u_{3p} u_0$, where $\{u_{3i-2}, u_{3i-1}, u_{3i+1}, u_{3i+2}\}=N_G(u_{3i})$ and $\{u_{3i-3}, u_{3i}\}=N_G(u_{3i-1}) \cap N_G(u_{3i-2})$ for each $i \in [p]$, where $[p]:=\{1, 2, \ldots, p\}$.

Observe that $M_4 \cong K_4$ and $M_7$ is known as the Moser spindle.

\begin{figure}[H]
    \centering
    \includegraphics[width=0.4\linewidth]{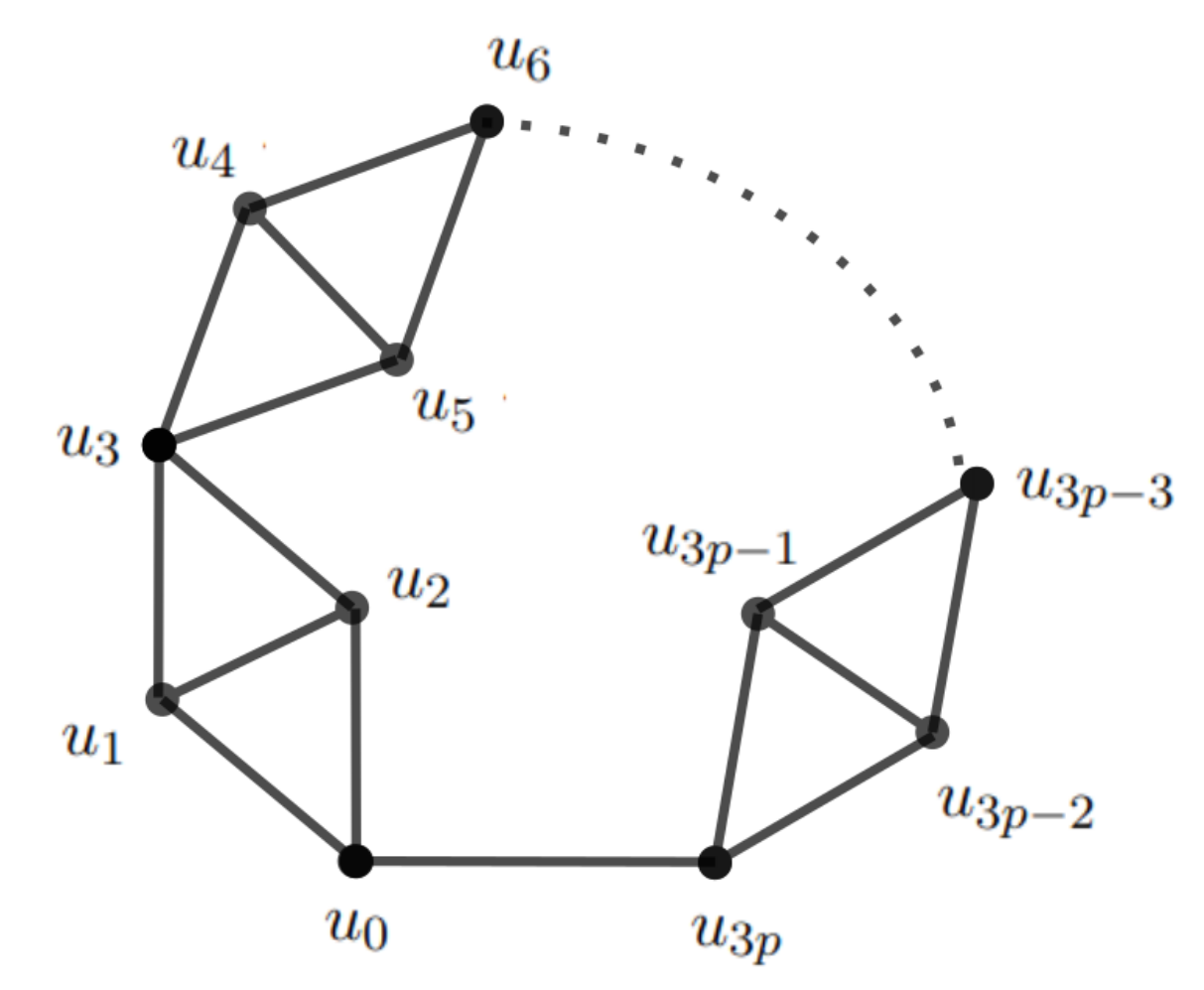}
    
    \caption{The spindle graph $M_{3p+1}$. It could be also consider as the clique  expansion $C_{2p+1}[2,1,2,1,...,2,1,1]$ for $p\geq 2$.}\label{spindler}
\end{figure}

\begin{proposition}[\cite{clawdiamondnet}]
    The graph $M_{3p+1}$ is not 3-colorable for every $p \geq 1$. 
\end{proposition}

Since the $3$-colorability problem is NP-complete for claw-free graphs and $K_3$-free graphs (cf.~\cite{GJPS17}), it is also NP-complete for $bull$-free graphs.
The following theorem in~\cite{clawdiamondnet} and \cite{stara} have motivated our research.

\begin{theorem}[\cite{clawdiamondnet}]
    Let $G$ be a connected $(bull, claw)$-free graph. Then one of the following holds
    \begin{enumerate}[label=(\roman*)]
        \item $G$ contains  $W_5$ or
        \item $G$ contains a spindle graph $M_{3i+1}$ for some $i\geq1$ or 
        \item $G$ is 3-colorable.
    \end{enumerate}
\end{theorem}

\begin{theorem}[\cite{stara}]\label{thm:bullchair} 
Let $G$ be a connected $(bull, chair)$-free graph. Then 
\begin{enumerate}[label=(\roman*)]
    \item $G$ contains an odd wheel or
    \item $G$ contains a spindle graph $M_{3i+1}$ for some $i\geq 1$ or
    \item $G$ is 3-colorable.
\end{enumerate}
\end{theorem}

The goal of this paper is to study $k$-colorability of $(bull, H)$-free graphs.
For $k\geq 4$, we will assume without loosing generality that $\delta(G) \geq 4$, since otherwise $G$ can be reduced by removing vertices of degree less than $4$ (its coloring is trivial). The following are our main results. The first theorem provides necessary and sufficient conditions for any clique expansion of an odd cycle to be $k$-colorable.

\begin{theorem}\label{cliqueexpansion}
    Let $n\geq 1$, $k\geq3$ and $G$ be a clique expansion $C_{2n+1}[k_1, \ldots, k_{2n+1}]$. Then $G$ is $k$-colorable if and only if the following two conditions are satisfied (all indices are taken modulo $k$): 
    \begin{enumerate}[label=(\roman*)]
        \item \label{1.} $\forall i\in [2n+1]\;\; k_i+k_{i+1} \leq k$;  
        \item \label{2.} $\sum_{i=1}^{2n+1} k_i \leq nk$.
    \end{enumerate}
\end{theorem}

And, we can observe an easy corollary of it using Theorem \ref{thm:bullclaw2} and Theorem \ref{thm:claw}.

\begin{corollary}\label{thm:kkol}
Let $G$ be a connected $(bull, claw)$-free graph containing an induced cycle of length $p\geq 7$. If $\alpha(G)\geq 3$, then $G$ is $k$-colorable or $G$ contains $K_{k+1}$ or $G$ is a clique expansion $C_{p}[k_1, \ldots, k_{p}]$ such that there exists $i\in [p]$ such that $k_i+k_{i+1} > k$ or  $\sum_{i=1}^{p} k_i > nk$, where all indices are taken modulo $k$.
\end{corollary}

Next, we obtain the full characterization of all non $4$-colorable connected  $(bull,claw)$-free graphs, and $(bull,chair, C_5)$-free graphs. 

\begin{theorem}\label{thm:bullclaw} 
Let $G$ be a connected $(bull, claw)$-free  graph and $i\geq 1$. Then 
\begin{enumerate}[label=(\roman*)]
    \item $G$ contains $\overline{C}_7 \oplus K_1$ or
    \item $G$ contains $C_5 \oplus K_2$ or
    \item $G$ contains $M_4 \oplus K_1$ or $M_7 \oplus K_1$ or
    \item $G$ contains $C_{2i+1}[2, 2, \ldots, 2, 1, 3, 1, 3, \ldots, 1]$ or $C_{2i+1}[2, 2, \ldots, 2, 1]$ or
    \item $G$ contains an induced cycle $C_5$ and $|V(G)|>8$ or
    \item $G$ is 4-colorable. 
\end{enumerate}
\end{theorem}

\begin{theorem}\label{thm:bullchairC5} 
Let $G$ be a connected $(bull, chair, C_5)$-free  graph and $i \geq 1$. Then 
\begin{enumerate}[label=(\roman*)]
    \item $G$ contains $\overline{C}_7 \oplus K_1$ or
    \item $G$ contains $C_{2i+1} \oplus K_2$ or
    \item $G$ contains $M_{3i+1} \oplus K_1$ or
    \item $G$ contains $C_{2i+1}[2, 2, \ldots, 2, 1, 3, 1, 3, \ldots, 1]$ or $C_{2i+1}[2, 2, \ldots, 2, 1]$ or
    \item $G$ is 4-colorable. 
\end{enumerate}
\end{theorem}

Finally, we present a full characterization of all non $5$-colorable connected $(bull, claw, C_5)$-free graphs.

\begin{theorem}\label{thm:bullclawC5} 
Let $G$ be a connected $(bull, claw, C_5)$-free  graph. Then 
\begin{enumerate}[label=(\roman*)]
    \item $G$ contains $K_6$ or 
    \item $G$ contains $\overline{C}_7 \oplus K_2$ or
    \item $G$ contains $\overline{C}_9 \oplus K_1$ or
    \item $\alpha(G)=2$ and $|V(G)| \geq 11$ or
    \item $\alpha(G)=2$ and $\Delta(G) \geq 9$ or
    \item $G$ is a clique expansion $C_{2n+1}[k_1, \ldots, k_{2n+1}]$ with  $k_1 + \ldots +k_{2n+1} - 5n > 0$ or
    \item $G$ is 5-colorable. 
\end{enumerate}
\end{theorem}

\section{Preliminary results}\label{sec:preliminary}


We recall that a \textit{hole} in a graph $G$ is an induced cycle of
length at least $4$, and an \textit{antihole} in $G$ is an induced subgraph whose
complement is a cycle of length at least~$4$.
A hole (antihole) is \textit{odd} if it has an odd number of vertices.
As the main tool for proving Theorem~\ref{thm:bullclaw}
we will use the well-known Strong Perfect Graph Theorem
shown by Chudnovsky et al.~\cite{ChRST06}.

\begin{theorem}[Chudnovsky et al.~\cite{ChRST06}]
\label{tSPGT}
A graph is perfect if and only if it contains neither an odd hole
nor an odd antihole as an induced subgraph.
\end{theorem}

\subsection{Independence number in $claw$-free graphs} 

The following two theorems have been shown in \cite{alpha3} and Lemma \ref{lem:ben} is due to Ben Rebea.

\begin{theorem}\label{thm:bullclaw2}~\cite{alpha3}
    Every connected $(bull, claw)$-free graph $G$ such that $\alpha (G) \geq 3$ is perfect or is a clique expansion of an odd cycle of length at least $7$.
\end{theorem} 

\begin{theorem}\label{thm:claw}\cite{alpha3}
    Let $G$ be a connected $(bull, claw)$-free graph. Then
    \begin{enumerate}[label=(\roman*)]
        \item if $G$ contains an independent set of size 3, then $G$ is $C_5$-free.
        \item if $G$ contains an induced cycle of length $k$ with $k \geq 6$, then $G$ is a clique expansion of $C_k$.
    \end{enumerate}
    
\end{theorem}

\begin{Lemma}\label{lem:ben}\cite{benrebea}
   If $G$ is a $claw$-free graph such that $\alpha (G)\geq 3$ and $G$ contains an odd antihole, then $G$ contains induced $C_5$.
\end{Lemma}

Combining Theorem \ref{thm:claw} and Lemma \ref{lem:ben} we obtain the following corollary.

\begin{corollary}\label{lem:alpha2}
Let $G$ be a connected $(bull, claw)$-free graph. If $G$ contains an odd antihole, then $\alpha (G) = 2.$ \qed
\end{corollary}

\section{Lemmas}

Let $G$ be a $(bull, chair)$-free graph such that $G$ contains an induced odd antihole $\overline{Q}=v_1\ldots v_p$.

\begin{Lemma}\label{notwocons}
    If a vertex $w \in G\setminus \overline{Q}$ is adjacent to $\overline{Q}$, then $w$ has no two consecutive non-neighbors in $\overline{Q}$.
\end{Lemma}
\begin{proof}
      Suppose $w$ has $\ell$ consecutive non-neighbors $v_i, \ldots, v_{i+\ell-1}$, where $1<\ell<7$, and $wv_{i-1} \in E(G)$. Then,  the set $\{w, v_{i-1}, v_{i}, v_{i+1}, v_{i+2}\}$ induces a $chair$, if $w$ is not adjacent to $v_{i+2}$, or a $bull$, if it is (see Fig. \ref{pic:anti_a}).
\end{proof}

\begin{Lemma}\label{noN2}
    Let a vertex $w \in N_2(\overline{Q})$ be adjacent to a vertex $w' \in N(\overline{Q})$. Then $w'$ is adjacent to all vertices of $\overline{Q}$.
\end{Lemma}
\begin{proof}
    Suppose $uw \in E(G)$. By Lemma \ref{notwocons}, $w$
    must have at least two consecutive neighbors $v_i, v_{i+1}$ on $\overline{Q}$ (since $\overline{Q}$ is odd). Without loss of generality $v_{i-1}w \notin E(G)$. Then the set $\{u, w, v_{i-1}, v_i, v_{i+1}\}$ induces a $chair$ (see Fig. \ref{pic:anti_b}).
\end{proof}

    \begin{figure}[htb]
    \centering    
    \begin{subfigure}{0.35\textwidth}
    \centering  
    \includegraphics[width=0.8\linewidth]{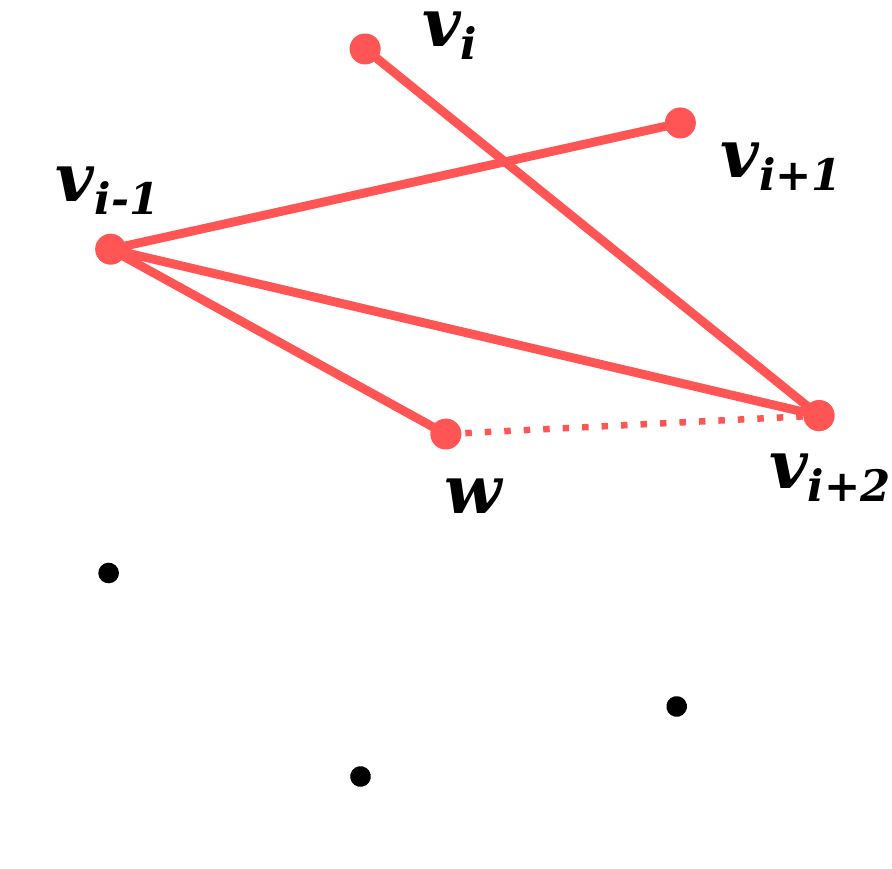}
    \caption{\label{pic:anti_a}}    
    \end{subfigure}
    \qquad
    \begin{subfigure}{0.35\textwidth}
    \centering  
    \includegraphics[width=0.8\linewidth]{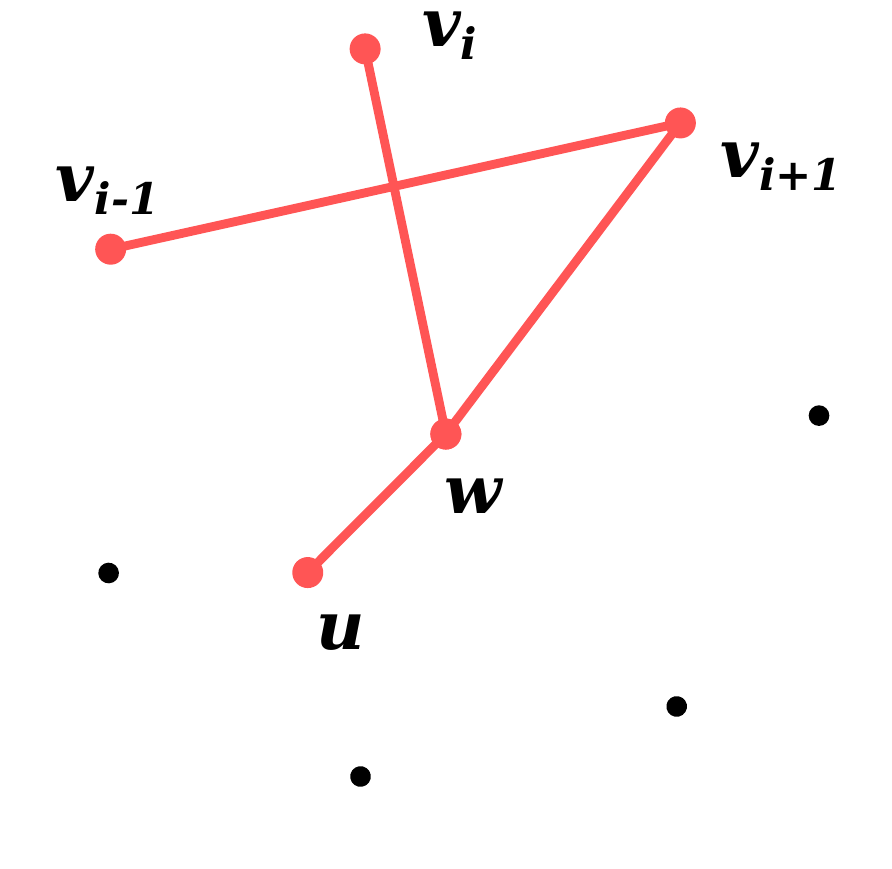}
    \caption{\label{pic:anti_b}}    
    \end{subfigure}
    \caption{Induced subgraphs constructed in the proofs of Lemma \ref{notwocons} and \ref{noN2}.}\label{pic:anti}
\end{figure}

\begin{Lemma}\label{d4}
    Let $w \in N(\overline{Q})$ be a vertex such that $d_{\overline{Q}}(w)=4$. Then $G$ contains an induced cycle $C_5$.
\end{Lemma}
\begin{proof}
    Suppose that $\overline{Q} \neq \overline{C_5}=C_5$ and $d_{\overline{Q}}(w)=4$. By Lemma \ref{notwocons}, $w$ must have two neighbors $v_i, v_{i+1}$ such that $wv_{i-1}, wv_{i+2}\notin E(G)$. Then the set of vertices $\{w, v_i, v_{i+2}, v_{i-1}, v_{i+1}\}$ induces $C_5$.
\end{proof}

\section{Proof of Theorem \ref{cliqueexpansion}}

By the definition of the clique expansion, the condition \ref{1.} is necessary. It is not difficult to see that condition \ref{2.} is necessary as well. Namely, let us denote by $K^i$ the $i$-th clique of the expansion and assume, without loosing generality, that $K^1$ is colored with colors $1, \ldots, k_1$. Of course, $G$ is $k$-colorable if and only if there exists such a coloring of $G-K^{2n+1}$ that at least $s$ colors from $\{1, \ldots, k_1\}$ are repeated on $K^{2n}$, where 
    $$
    k_1 + k_{2n} - s +k_{2n+1} = k
    $$
    and therefore $s=k_1+k_{2n}+k_{2n+1}-k$. Let us say that if $s \leq 0$, then our task becomes trivial - the $CC$ algorithm starting from $K^1$ and cut at $K^{2n}$ can be completed to a proper coloring. Thus, we can assume that $k_i +k_{i+1} +k_{i+2}> k$ for any $i$.
    
   How many colors from $K^1$ can be repeated at most on $K^{2n}$? Consider a simplified problem, how many colors from $K^i$ can be repeated on $K^{i+3}$? If $S \subset c(K^i)$ is a set of colors that we would like to repeat (if possible) on $K^{i+3}$, then $K^{i+2}$ contains at least $|S|+k_{i+2}-(k-k_{i+1})$ colors from $S$ (note that we always have $k_{i+2}\geq |S|+k_{i+2}-(k-k_{i+1})$, otherwise $k < |S| + k_{i+1} \leq k_i +k_{i+1}$, which contradicts the assumption). Therefore, $K^{i+3}$ contains no more than $|S| - (|S|+k_{i+2}-(k-k_{i+1}))= k - k_{i+1}- k_{i+2}$ colors from $S$.

    Thus, on $K^4$ we can repeat at most $k-k_2-k_3$ desired colors, on $K^6$ additionally no more than $k-k_4-k_5$, and so on. Finally, on $K^{2n}$ we can repeat no more than $(n-1)k- k_2 - \ldots- k_{2n-1} $ colors from $\{1, \ldots, k_1\}$. In order to obtain a proper coloring, we need at least $s$ colors repeated, so we obtain a necessary condition
    $$
    s = k_1 + k_{2n} +k_{2n+1}-k \leq (n-1)k - k_2 - \ldots - k_{2n-1}
    $$
    equivalent to \ref{2.}.

   It is easy to show that conditions \ref{1.} and \ref{2.} are sufficient. Let us start coloring cliques with $k$-CC algorithm and let $l \in \{1, \ldots, n\}$ be the smallest index such that $c(K^{2l})$ contains colors $m+1, \ldots, m+p$, where $p \geq s$ and $m+p \leq k_1$. Such an index exists by the condition \ref{2.}. Then, we can color every clique $K^{2j+1}$, $l \leq j <n $, with consecutively increasing colors $1, 2, \ldots, m$ and then (if $k_{2j+1}>m$) with consecutively decreasing colors $k, k-1, \ldots$. Every clique $K^{2j}$, $l<j \leq n$, we color with consecutively increasing colors $m+1, m+2, \ldots$. It is easy to see that $K^{2n}$ has at least $p$ common colors with $K^1$, thus, we can always color $K^{2n+1}$ with remaining colors.

\section{Proof of Theorem \ref{thm:bullclaw}}

Let us point out that for perfect graphs Theorem \ref{thm:bullclaw} is trivially true. Therefore, we can restrict our attention to non-perfect graphs. Let $G$ be a $(bull, claw)$-free graph. If $G$ is non-perfect, it must contain an odd antihole or an odd cycle of length at least $5$.

Note that by Corollary \ref{lem:alpha2} this leaves us with only two possible cases.
\subsection{$G$ contains an induced odd cycle of length at least $7$ and $\alpha(G)\geq 3$}

Let us recall that by Theorem \ref{thm:bullclaw2}, the graph $G$ is a clique expansion of the cycle. The easy corollary of Theorem \ref{cliqueexpansion} is that if $G^*=C[k_1, \ldots, k_p]$ is a clique expansion of an odd cycle of length at least 7, then either  $G^*$ contains $K_5= M_4\oplus K_1$ or $G^*=C[2,2, \ldots, 2,1,3,1, 3, \ldots, 1]$ or $G^*=C[2, \ldots, 2, 1]$, or $G^*$ is $4$-colorable. 

\subsection{$G$ contains an odd antihole, so $\alpha(G)=2$}

Let $\overline{Q} = v_1 v_2 \ldots v_p$ be an odd antihole and $N_2(\overline{Q})$ be a set of all the vertices of distance~$2$ from $\overline{Q}$. Let us point out two simple facts.

\begin{fact}\label{Qdom}
    $\overline{Q}$ is a dominating set in $G$.
\end{fact}
\begin{proof}
Suppose there exists $w \in N_2(\overline{Q})$. Then we have an independent set $\{w, v_1, v_2\}$ of $3$ vertices, which contradicts Corollary \ref{lem:alpha2}.
\end{proof}

Now, using similar argument as above we can prove as follows.

\begin{fact}\label{konkr}\label{konkr2}
     Let $w, w' \in N(Q)$ and $N(w)\cup N(w') \neq Q$. Then $ww' \in E(G)$. \qed
\end{fact} 

Note that if $p \geq 11$, then $\overline{Q}$ contains $K_5$, but case $(iii)$ of Theorem \ref{thm:bullclaw}. Next, if $p=9$, then the graph contains $C[1,3,1,3,1]$, but then case $(iv)$ of Theorem \ref{thm:bullclaw}. So, we assume that $p=5$ or $p=7$ and we show the coloring of the graph $G$. 

\hspace{.5 cm}

First, assume $p=7$. We will show that if the graph $G$ is not $4$-colorable, then it does contain one of the exceptional subgraphs.

By Lemma \ref{notwocons} we know that for any $w \in N(\overline{Q})$ it holds $d_{\overline{Q}}(w) \geq 4$.  Of course, if there is a vertex $w \in V(G)$ such that $N_{\overline{Q}}(w)=\overline{Q}$, then we have exceptional graph $\overline{C}_7 \oplus K_1$. Moreover, if there is a vertex $w$ with $d_{\overline{Q}}(w)=4$, then by Lemma \ref{d4} we have an induced $C_5$, which was considered in the previous case. 

Thus, assume $d_{\overline{Q}}(w)\in \{5, 6\}$. If $N(\overline{Q})=\{w\}$, then $G$ is obviously $4$-colorable. Let $w' $ be another vertex in $N(\overline{Q})$. If $N_{\overline{Q}}(w) \cup N_{\overline{Q}}(w')=\overline{Q}$, then $G$ contains $C[1,3,1,3,1]$. Otherwise we must have $ww' \in E(G)$, and $G$ contains a complete graph $K_5$.

\hspace{.5 cm}

Finally let $p=5$. We can assume that $|V(G)| \leq 8$, otherwise we have point $(v)$ of Theorem \ref{thm:bullclaw}. By Lemma \ref{notwocons} we know that for any $w\in N(\overline{Q})$ we have $d_{\overline{Q}}(w)\in \{3, 4, 5\}$ and if $d_{\overline{Q}}(w)=3$, then non-neighbors of $w$ are non-consecutive. Let us define the following sets:
\begin{itemize}

\item[] $A_i= \{ v \in N(\overline{Q}) : N_{\overline{Q}}(v)=\{v_i, v_{i+1}, v_{i+3}\} \}$.

\item []$B_i= \{ v \in N(\overline{Q}) : N_{\overline{Q}}(v)=\{v_i, v_{i+1}, v_{i+2}, v_{i+3}\} \}$.  

\item []$C= \{ v \in N(\overline{Q}) : N_{\overline{Q}}(v)=\overline{Q} \}$.

\end{itemize}

Let also $A = \bigcup_{i=1}^p A_i$ and $B = \bigcup_{i=1}^p B_i$. By Fact \ref{Qdom}, we have $V(G) = Q \cup A \cup B \cup C$. 

\noindent \textit{Case 1.} A set $C\neq \emptyset$. Let $w\in C$. If there is another vertex $w' \in C$, then $ww' \notin E(G)$ (otherwise we have an exceptional graph $C_5 \oplus K_2$). Thus, we can give $w$ and $w'$ the same color $1$. If there is a third vertex $w'' \in N(Q)$, then note that $w'' \notin C$ (otherwise we have a $claw$, if $w,w', w''$ are not adjacent, or $C_5 \oplus K_2$, if they are). So we can color with $2$ the vertex $w''$ and one of its non-neighbors on the cycle $Q$. The rest of the cycle we color with $3$ and $4$. If $w', w'' \in A\cup B$, then note that either $ww' \notin E(G)$ or $ww'' \notin E(G)$ or $N_Q(w')\cup N_Q(w'') \neq Q$ (otherwise the graph $G$ contains $M_7 \oplus K_1$). Assume $ww' \notin E(G)$. Then we color $w, w'$ with $1$, $w''$ and one of its non-neighbors on the cycle with $2$ and the rest of the cycle with $3, 4$. Analogously for $ww'' \notin E(G)$. Assume $ww', ww'' \in E(G)$. As we said before, we have 
$N_Q(w')\cup N_Q(w'') \neq Q$. But then the vertices $w', w''$ must be adjacent and the graph $G$ contains $K_5$.\\
\textit{Case 2.} A set $C = \emptyset$. Let $w, w', w'' \in A \cup B$. Then each of those vertices has at least 1 non-neighbor on $Q$. Due to our assumption that $\delta(G) \geq 4$, the non-neighborhoods are disjoint. If possible, we take three vertices $u, u', u''$, non-neighbors of $w, w', w'$ respectively, such that only two of $u, u', u''$ are adjacent. We color $w, w', w''$ and their non-neighbors with colors $1, 2, 3$, respectively. The remained vertices of $Q$ are non-adjacent, so we can color them with $4$. If such a triple of non-neighbors does not exist, it means $w, w', w'' \in B$ and their two common neighbors are adjacent. Thus, without loss of generality, $ww' \notin E(G)$ (or we have $K_5$). Then we color $w, w'$ with $1$, $w''$ and its non-neighbor with $2$ and the rest of the cycle with $3$  and $4$, what finishes the proof.


\section{Proof of Theorem \ref{thm:bullchairC5}}

As before, the result is obvious for perfect graphs, so using Theorem \ref{tSPGT} we can split the proof into two cases -- when the graph $G$ contains an induced odd antihole and when it contains an induced odd hole. In this theorem, we also forbid $C_5$, so both hole and antihole must be of length at least $7$. 

\subsection{$G$ contains an odd antihole}

Let us assume the graph $G$ contains an odd antihole $\overline{Q}=v_1\ldots v_p$. As we have mentioned in the proof of Theorem \ref{thm:bullclaw}, if $p\geq11$, then $G$ contains $K_5$, and if $p=9$, then $G$ contains $C[1,3,1,3,1]$. So let $p=7$. We will see that our conclusions will be identical as in the respective part of the proof of Theorem \ref{thm:bullclaw}.

By Lemma \ref{noN2}, if there is a vertex $u\in N_2(\overline{Q})$, then it must have a neighbor $w$ such that $N_{\overline{Q}}(w)=\overline{Q}$. But if such a $w$ exists, then $G$ contains an exceptional subgraph $\overline{C}_7 \oplus K_1$. Moreover, by Lemma \ref{d4}, if $d_{\overline{Q}}(w)=4$, then $G$ contains $C_5$. So we can assume $N_2(\overline{Q})=\emptyset$ and $d_{\overline{Q}}(w) \in \{5,6\}$ for any $w\in N(\overline{Q})$ and we finish the coloring as in the proof of Theorem \ref{thm:bullclaw}.

\subsection{$G$ contains induced odd cycle of length at least $7$}

Let $Q=v_1\ldots v_p$ be an induced odd cycle of length at least 7. We will prove that if a $(bull, chair)$-free graph $G$ does contain such a cycle, then it satisfies some useful structural properties. A similar, but more general structural analysis of $(bull, chair)$-free graphs can be found in \cite{kkol}. 

All indices will be taken modulo $p$.

 \begin{fact}
     Let $w \in N(Q)$. Then either $N_Q(w)=\{v_{i-1}, v_i, v_{i+1}\}$ for some $i\in \{1,\ldots, p\}$ or $N_Q(w)=~Q$. 
 \end{fact}
 \begin{proof}
     Let $v_k, \ldots, v_{k+\ell -1}$ be the longest sequence of consecutive neighbors of $w$ on the cycle and suppose $\ell < p$. We will show that in this case $\ell=3$ and $v_k, v_{k+1}, v_{k+2}$ are the only neighbors of $w$ on the cycle. 

     Firstly, suppose that $\ell=1$, that is, $w$ has no consecutive neighbors on $Q$. Then, since~$Q$ is odd, the graph $G$ must contain an induced $chair$. 

     Suppose now that $\ell=2$. Then the set $\{v_{k-1}, v_k, w, v_{k+1}, v_{k+2}\}$ induces a $bull$. 

\begin{figure}[htb]
    \centering    
    \begin{subfigure}{0.4\textwidth}
    \centering  
    \includegraphics[width=1.1\linewidth]{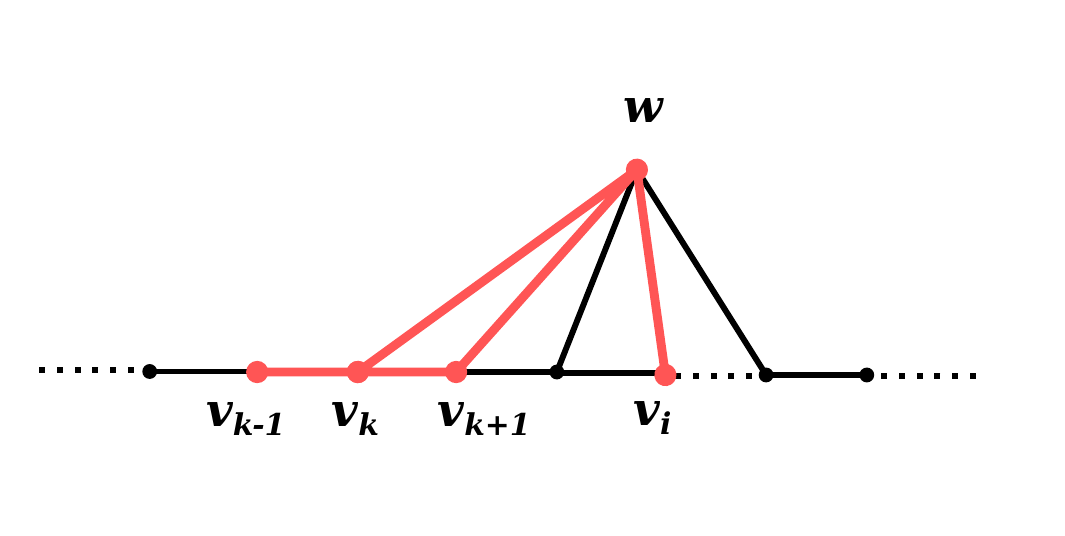} 
    \caption{\label{pic:claim_expansion_a}}
    \end{subfigure}
    \begin{subfigure}{0.4\textwidth}
    \centering  
    \includegraphics[width=1.1\linewidth]{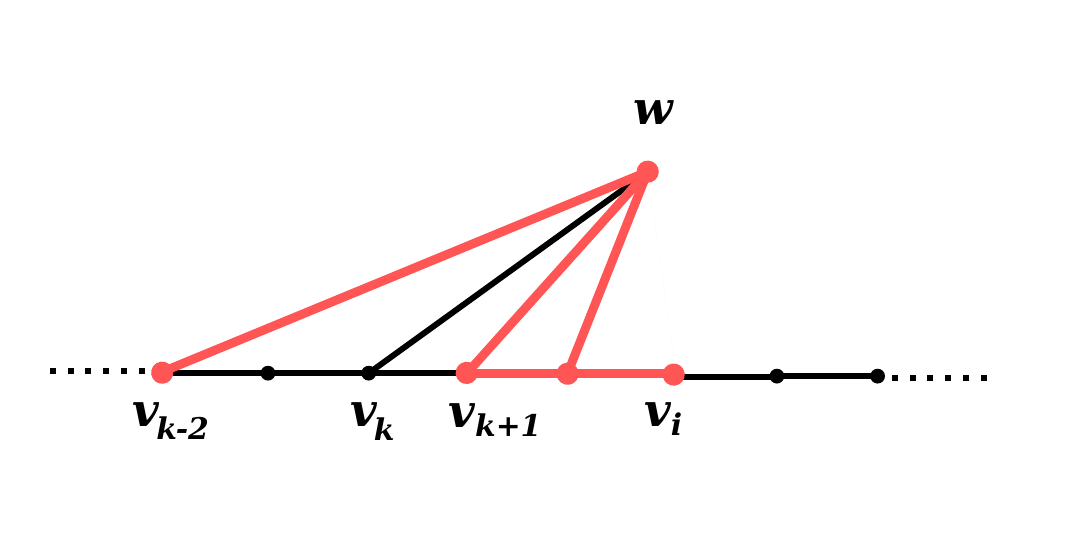}
    \caption{\label{pic:claim_expansion_b}}
    \end{subfigure}
    \caption{Induced subgraphs constructed in the proof of Fact \ref{thm:bullchair}}.\label{pic:thm_bullchair}
\end{figure}
    
         Finally, suppose $\ell \geq 3$ and $w$ has a neighbor $v_i$ among $\{v_4, \ldots, v_{k-2}\}$. Then the set $\{v_{k-1}, v_k, w, v_{k+1}, v_i\}$ (if there is $v_i$ such that $i\neq k-2$) -- see Figure \ref{pic:claim_expansion_a}) or the set $\{v_{k-2}, w, v_{k+1}, v_{k+2}, v_{k+3}\}$ (otherwise) induces a $bull$ -- see Figure \ref{pic:claim_expansion_b}. 
\end{proof}

\hspace{.5cm}

Now we can define the following sets:

\begin{itemize}
\item []$A_i= \{ v \in N(Q) : N_{Q}(v)=\{v_{i-1}, v_i, v_{i+1}\} \}$, and $A = \bigcup_{i=1}^p A_i$. 
\item []$D= \{ v \in N(Q) : N_{Q}(v)=Q \}$.
\end{itemize}

\begin{fact}
    $G[Q \cup A]$ is a clique expansion of the cycle $Q$.
\end{fact}
\begin{proof}
    Let $w, w' \in A_i\setminus\{v_i\}$. Of course, if $ww' \notin E(G)$, then the set $\{v_{i-3}, v_{i-2}, v_{i-1}, w, w'\}$ induces a $chair$. 
    
    Let now $w' \in A_{i+1} \setminus \{v_{i+1}\}$ and let $ww' \notin E(G)$. Then the set $\{v_{i-2}, v_{i-1}, w, v_i, w'\}$ induces a $bull$. 

    Finally, let $w' \in A_j$, where $|j-i|\geq 2$, and let $ww' \in E(G)$. By symmetry we can assume $j \leq i-4$.   Then the set $\{v_{i-2}, v_{i-1}, v_i, w, w'\}$ induces a $bull$. 
\end{proof}

\begin{fact}
    $D$ separates $Q$ and $G\setminus(Q \cup A\cup D)$. Moreover, if $w \in D$ is adjacent to a connected component $C$ of  $G\setminus(Q \cup A\cup D)$, then $w$ dominates $C$.
\end{fact}
\begin{proof}
  To prove the first part of the claim we need only to show that vertices from $A$ cannot be adjacent to the second neighborhood of $Q$. And this is obvious, since if a vertex $v \in A_i$ has a neighbor $u$ in the second neighborhood of $Q$, then $\{v_{i-2}, v_{i-1}, v_i, v, u\}$ induces a $bull$. 
  
  Suppose now $w\in D$ is adjacent to a connected component $C$ of $G\setminus (Q \cup A)$, but does not dominate $C$. Thus, there exist $u, u' \in C$ such that $wu, uu' \in E(G)$ but $wu' \notin E(G)$. But then the set $\{u', u, w, v_1, v_3\}$ induces a $chair$. 
\end{proof}

Note that we can assume, without loosing generality, that $D$ is an independent set (otherwise we have an exceptional subgraph $C_{2k+1} \oplus K_2$). Thus, the graph $G$ is 4-colorable if and only if 
\begin{enumerate}[label=(\roman*)]
    \item $G[Q \cup A]$ is 3-colorable and
    \item $G\setminus (Q \cup A \cup D)$ is 3-colorable. 
\end{enumerate}
By Theorem \ref{thm:bullchair} we know that $G\setminus D$ is 3-colorable if and only if it does not contain an odd wheel or a spindle graph $M_{3i+1}$. Every connected component $C$ of $G\setminus D$ is dominated by some vertex $w$ of $D$, so if $C$ does contain odd wheel, then $G$ contains $C_{2k+1}\oplus K_2$, and if $C$ contains $M_{3i+1}$, then $G$ contains $M_{3i+1}\oplus K_1$. 

This completes the proof of Theorem \ref{thm:bullchairC5}. 
\section{Proof of Theorem \ref{thm:bullclawC5}}

\medskip

The proof is again similar to the proof of Theorem \ref{thm:bullclaw}. Let us point out that for perfect graphs the theorem is trivially true. Therefore, we can restrict our attention to non-perfect graphs. A non-perfect $(bull, claw, C_5)$-free graph $G$ must contain an odd antihole or an induced cycle of length $p \geq 7.$ 





With Corollary \ref{lem:alpha2} this leaves us with only two possible cases.
\subsection{$G$ contains an induced odd cycle of length at least $7$ and $\alpha\geq 3$}

Let us recall that by Theorem \ref{thm:bullclaw2} our graph $G$ is a clique expansion of the cycle. Now by 
Theorem \ref{cliqueexpansion} $G$ is 5-colorable or 
$k_i+k_{i+1} \geq 6$ for some integer $i\in\{ 1, \ldots,  2n+1\}$ and $G$ contains $K_6$ or 
$k_1 + \ldots +k_{2n+1} - 5n > 0$. This is statement $(i)$ and $(vi)$.

\subsection{$G$ contains an odd antihole and $\alpha(G)=2$}

For a graph $G$ let $\beta_0(G)$ denote its matching number. The following fact is well known.

\begin{fact}\label{ChromCompl}
Let $G$ be a graph with $\alpha(G)=2.$ Then $\chi(G) = n - \beta_0(\overline{G}) \geq \frac{n}{2}.$ \qed
\end{fact} 

Let $\overline{Q} = v_1 v_2 \ldots v_p v_1$ be an odd antihole. Note that if $p \geq 13$ then $\overline{Q}$ contains $K_6$ and if $p=11$, then the graph contains $C[1,4,1,4,1]$. So we may assume that $p \in \{7,9\}.$ If $|V(G)| \geq 11,$ then $\chi(G) \geq 6$ 
by Fact \ref{ChromCompl}. This is statement $(iv)$. So we may assume that $|V(G)| \leq 10.$

First we consider the case $p=9.$ If $G \cong \overline{C}_9 \oplus K_1,$ then $\chi(G)=6.$ This is statement $(iii)$. If $G \ncong \overline{C}_9 \oplus K_1,$ then $\chi(G)=5.$

Next we consider the case $p=7.$ If $7 \leq |V(G)| \leq 8,$ then $G \subset W_7$, hence $G$ is 5-colorable. Now consider 
$|V(G)|=9.$ If $G \cong \overline{C}_7 \oplus K_2,$ then $\chi(G)=6.$ This is statement $(ii)$. If $G \ncong \overline{C}_7 \oplus K_2,$ then $\chi(G)=5$ by Fact \ref{ChromCompl}. 

Finally, consider $|V(G)|=10.$ If $\Delta(G)=9,$ then $G$ is not 5-colorable. This is statement $(v)$. 
If $\beta_0(\overline{G})=5,$ then $G$ is 5-colorable by Fact \ref{ChromCompl}.
Hence we may assume that $\Delta(G) \leq 8$ and $\beta_0(\overline{G}) \leq 4.$ Let $H:= G - \overline{Q}.$ We first show the following fact.

\begin{fact}\label{Hcomplete}
   The graph $H$ is complete.
\end{fact}
\begin{proof}
Let $C= \{ v \in N(\overline{Q}) : N_{\overline{Q}}(v)=\overline{Q} \}$. Let $H = \{w_1,w_2,w_3\}.$ Suppose $H$ is not complete 
and let $w_1w_2 \notin E(G).$ Then $w_3 \in C.$ Since $\Delta(G) \leq 8,$ we may assume that $w_1w_3 \notin E(G).$ Hence $w_2 \in C$ as well. Now $w_2w_3 \in E(G)$ and so $G$ contains $\overline{C}_7 \oplus K_2,$ This is statement $(ii)$.
\end{proof}

Since $\Delta(G) \leq 8$ we obtain $d_{\overline{Q}}(w) \leq 6$ for any $w \in V(H).$ 
By Lemma \ref{notwocons} we know that for any $w \in V(H)$ we have $d_{\overline{Q}}(w) \geq 4$. Moreover, if 
$d_{\overline{Q}}(w) = 4$, then by Lemma \ref{d4} we have induced $C_5$, a contradiction. Hence, we may assume that 
$5 \leq d_{\overline{Q}}(w) \leq 6$ for any $w \in V(H).$ Moreover, we obtain the following fact. It can be proven in exactly the same way as Lemma \ref{d4}.

\begin{fact}\label{antiholeC5}
If $d_{\overline{Q}}(w) = 5$ for a vertex $w \in V(H),$ then there exists some integer $i$ such that $wv_i, wv_{i+2} \notin E(G).$     
\end{fact}

\begin{fact}\label{antiholematching}
If $w_1v_i, w_2v_{i+1}, w_3v_{i+2} \notin E(G)$ or $w_1v_i, w_2v_{i+1}, w_3v_{i+4} \notin E(G),$ then $G$ is 5-colorable.   
\end{fact}

\begin{proof}
Observe that in these two cases we have $\beta_0(\overline{G})=5$, and $G$ is 5-colorable by Fact \ref{ChromCompl}.
\end{proof}

Suppose first that $d_{\overline{Q}}(w_i) = 6$ for $i=1,2,3.$ Then $G$ contains  $K_6$ or $G$ is 5-colorable by 
Fact \ref{antiholematching}. Suppose next that $d_{\overline{Q}}(w_1) = 5.$ We may assume that $w_1v_i, w_1v_{i+2} \notin E(G).$ 
If $w_2v_{i+1} \notin E(G)$ or $w_3v_{i+1} \notin E(G),$ then $\overline{Q} - v_{i+1} + w_1$ is a $\overline{C_7}$ such that the corresponding 
subgraph $H$ is not complete, contradicting Fact \ref{Hcomplete}. So we may assume that $w_2v_{i+1}, w_3v_{i+1} \in E(G).$ Now if 
$v_{i+3}, v_{i+6} \in N(w_2) \cap N(w_3),$ then there is  $K_6.$ Hence we may assume that $w_2v_{i+3} \notin E(G).$ Then by Lemma \ref{notwocons} and Fact \ref{antiholeC5} we obtain $v_{i+4}, v_{i+6} \in N(w_2).$ Now by Fact \ref{antiholematching} we conclude that 
$v_{i+4}, v_{i+6} \in N(w_3)$ and find  $K_6.$ 

This completes the proof of Theorem \ref{thm:bullclawC5}.

\end{document}